\documentclass[11pt,a4paper]{amsart}

\usepackage{times,epsfig}
\usepackage{amsfonts,amscd,amsmath,amssymb,amsthm}

\usepackage{enumitem}
\usepackage{bbm}
\usepackage{ifthen} 

\usepackage{tikz}
\usepackage{xy} \xyoption{all}
\usepackage{color,graphics}

\usepackage{hyperref}

\numberwithin{equation}{section}

\newcommand{\uloopr}[1]{\ar@'{@+{[0,0]+(-4,5)}@+{[0,0]+(0,10)}@+{[0,0] +(4,5)}}^{#1}}
\newcommand{\uloopd}[1]{\ar@'{@+{[0,0]+(5,4)}@+{[0,0]+(10,0)}@+{[0,0]+ (5,-4)}}^{#1}}
\newcommand{\dloopr}[1]{\ar@'{@+{[0,0]+(-4,-5)}@+{[0,0]+(0,-10)}@+{[0, 0]+(4,-5)}}_{#1}}
\newcommand{\dloopd}[1]{\ar@'{@+{[0,0]+(-5,4)}@+{[0,0]+(-10,0)}@+{[0,0 ]+(-5,-4)}}_{#1}}

\newcommand{\luloop}[1]{\ar@'{@+{[0,0]+(-8,2)}@+{[0,0]+(-10,10)}@+{[0, 0]+(2,2)}}^{#1}}

\DeclareSymbolFont{SY}{U}{psy}{m}{n}
\DeclareMathSymbol{\emptyset}{\mathord}{SY}{'306}

\DeclareMathSymbol{\newtimes}{\mathbin}{SY}{'264}

\newcommand{\fF}{\mathfrak{F}}

\newcommand{\fA}{\mathfrak{A}}

\newcommand{\cA}{{\mathcal A}}
\newcommand{\cB}{{\mathcal B}}

\newcommand{\cF}{{\mathcal F}}
\newcommand{\cH}{{\mathcal H}}
\newcommand{\cI}{{\mathcal I}}

\newcommand{\cK}{{\mathcal K}}

\newcommand{\cR}{{\mathcal R}}

\newcommand{\cW}{{\mathcal W}}

\renewcommand{\1}{\mathbbm 1}

{\bf}{\it}
{\bf}{\it}
{\bf}{\it}
{\bf}{\it}

{\bf}{\it}

\newtheorem{theorem}{Theorem}[section]{\bf}{\it}
\newtheorem{proposition}[theorem]{Proposition}{\bf}{\it}
{\bf}{\it}
\newtheorem{lemma}[theorem]{Lemma}{\bf}{\it}
{\bf}{\it}
{\bf}{\it}
{\bf}{\it}

\theoremstyle{definition}
\newtheorem{definition}[theorem]{Definition}
\newtheorem{remark}[theorem]{Remark}
\newtheorem{examples}[theorem]{Examples}

\theoremstyle{plain}

\theoremstyle{definition}

\newlist{Fenumerate}{enumerate}{1}
\setlist[Fenumerate, 1]{label = \textbf{(F\arabic*)}, ref = {{\bf (F\arabic*)}}}

\newcommand{\R}{\mathbb{R}}

\newcommand{\N}{\mathbb{N}}
\newcommand{\C}{\mathbb{C}}

\newcommand{\BH}{{\mathcal{B}(\mathcal{H})}}
\newcommand{\RS}{\R^*}
\newcommand{\quadtext}[1]{\quad\text{#1}\quad}

\begin{document}

\title{A note on commutation relations and finite dimensional approximations}

\author[Fernando Lled\'{o}]{Fernando Lled\'{o}$^{1}$}
\address{Department of Mathematics, University Carlos~III Madrid,
  Avda.~de la Universidad~30, 28911 Legan\'{e}s (Madrid), Spain
  and Instituto de Ciencias Matem\'{a}ticas (CSIC - UAM - UC3M - UCM).}
\email{flledo@math.uc3m.es}

\author[Diego Mart\'{i}nez]{Diego Mart\'{i}nez $^{2}$}
\address{Mathematisches Institut, WWU M\"{u}nster, Einsteinstr. 62, 48149 M\"{u}nster, Germany.}
\email{diego.martinez@uni-muenster.de}

\thanks{{$^{1}$} Supported by Spanish Ministry of Economy and Competitiveness through the
Severo Ochoa Programme for Centres of Excellence in R\& D (SEV-2015-0554) and from the Spanish National Research Council, through the \textit{Ayuda extraordinaria a Centros de Excelencia Severo Ochoa} (20205CEX001).}

\thanks{{$^{2}$} Partly funded by research projects MTM2017-84098-P, Severo Ochoa SEV-2015-0554 and BES-2016-077968 of the Spanish Ministry of Economy and Competition (MINECO), Spain. Partly funded by the Deutsche Forschungsgemeinschaft (DFG, German Research Foundation) under Germany’s Excellence Strategy – EXC 2044 – 390685587, Mathematics Münster – Dynamics – Geometry – Structure; the Deutsche Forschungsgemeinschaft (DFG, German Research Foundation) – Project-ID 427320536 – SFB 1442, and ERC Advanced Grant 834267 - AMAREC}

\date{\today}


\begin{abstract}
In this article we show that the main C*-algebras describing the canonical commutation relations of quantum physics, i.e., the Weyl and resolvent algebras, are in the class of 
F\o lner C*-algebras, a class of C*-algebras admitting a kind of finite approximations of F\o lner type. In particular, we show that the tracial states of the resolvent algebra are uniform locally finite dimensional.
\end{abstract}

\maketitle

\section{Introduction}
\label{sec:1}

One of the most fundamental relations in quantum mechanics is the canonical commutation relation (CCR) between  
position and momentum, which when written in the most simple case (and taking $\hbar=1$) is
\begin{equation}\label{eq:ccr}
 QP-PQ=i\1\;.
\end{equation}
Heisenberg attributed this ``ingenious'' relation to Born in a famous letter to Pauli in 1925.
It soon became clear that neither matrices nor bounded operators on a Hilbert space can 
represent the relation (\ref{eq:ccr}) (cf. \cite{Wintner47,Wielandt49}). This fact opens several possibilities
to analyze mathematically these relations. A first natural development is to study the commutation relations (or more
generally commutators) of unbounded operators affiliated to finite von Neumann algebras, i.e., von Neumann algebras where
the identity $\1$ is a finite projection in the sense of Murray and von Neumann. In the case of the CCRs this restricts necessarily to 
type II$_1$ von Neumann algebras which are infinite dimensional but still allow some kind of finiteness through the existence of a normalized trace.
We refer to \cite{KL14,KLT20} for further motivations, results and an interesting historical account. 
Alternatively, the commutation relations can be reformulated in terms of more tractable bounded operators using bounded functions of the canonical variables 
(or fields) and encode the commutation relations in terms of them. There are two main options in the mathematical physics literature: the well known Weyl (or CCR) algebra
\cite{Manuceau68} generated by complex exponentials of the fields and the more recent resolvent algebra introduced by Buchholz and Grundling in \cite{BG08}.
The connection between the generators of the Weyl and resolvent algebras and the fields satisfying the commutation relations 
is obtained in terms of the notion of regular states. For example, 
in the case of the Weyl algebra $\cW(X,\sigma)$ associated to a symplectic vector space $(X,\sigma)$ a state $\omega$ is called \textit{regular} if the corresponding GNS-representation $\pi=\pi_\omega$ satisfies that the one-parameter groups 
\[
 \R\ni t\to \pi\Big(W(t\, f)\Big), \; \text{where} \; f\in X,
\]
are strongly continuous. Observe that $W(f)$ denotes the Weyl elements generating the C*-algebra.
Quantum fields appear as unbounded self-adjoint generators of the one-parameter groups associated to the Weyl elements in a regular representation. 
For the definition and further results concerning regular states in the context of the resolvent algebra see Section~4 in \cite{BG08}.

In a completely different context the theory of C*-algebras has been developed in the last decades focusing on different approximation patterns
in terms of finite structures and leaving the representation theory in the background. For instance, 
central notions in operator algebras like nuclearity, exactness or quasidiagonality (see, e.g., \cite{bBrown08,bBlackadar06}) are nowadays defined using nets of contractive completely positive (\textit{c.c.p.}) maps from the C*-algebra $\cA$ into matrices highlighting different approximation patterns of the algebra:
\[
 \varphi_n\colon\cA\to M_{k(n)}(\C).
\]
Quasidiagonality, for example, requires these maps to be asymptotically multiplicative and asymptotically isometric in the C*-norm. 
The class of F\o lner C*-algebras then weakens the convergence in the multiplicativity condition to a convergence in mean (see Definition~\ref{def:FA} for details), hence broadening the class of allowed approximations. 
These type of approximations show that the linear, the involutive and the order structure of the C*-algebra are respected by the approximating matrix algebras, while the product and the norm of $\cA$ are captured only asymptotically. This class of C*-algebras has several equivalent characterizations which can be useful in the description of quantum theory.
Indeed, let $\cA\subset\cB(\cH)$ be a unital C*-algebra modeling the quantum observables in a complex separable Hilbert space $\cH$ not containing the compact operators.
If $\cA$ is a F\o lner C*-algebra, then there is a net $\{\rho_n\}_n$ 
positive trace-class operators with $\mathrm{Tr}(\rho_n)=1$ and such that
\[
\lim_{n} \|\rho_n A-A\rho_n\|_1=0 \;\; \text{for every} \;\; A\in\cA,
\]
where $\|\cdot\|_1$ denotes the trace-class norm in $\cB(\cH)$. Moreover, it is well known that the net of states $\rho_n$ may be taken of finite range.
We refer to \cite{Bedos95,AL14,ALLW-2} for additional remarks and motivation concerning this class of algebras. Finally, F\o lner C*-algebras and, in particular 
quasidiagonal C*-algebras, admit a good spectral approximation behavior, e.g., weak convergence of spectral measures associated to selfadjoint elements (see, e.g., \cite{Arveson94,Bedos97,AL14,Lledo13,Brown06b}).

The aim of this article is to show that the main C*-algebras modeling the canonical commutation relations (the Weyl and the resolvent algebras algebras associated with a real symplectic vector space $(X,\sigma)$) are F\o lner C*-algebras and, hence, admit finite dimensional approximations in mean. Note that since
the commutation relations (either in terms of Weyl elements or resolvents (see Eqs.~(\ref{eq:weyl}) and (\ref{eq:resolvent}) below) involve products, one obtains a finite-dimensional approximation of the C*-algebra in terms of matrices and where the commutation relations appear only asymptotically. Recall that the Weyl and resolvent C*-algebras are both non-separable (except in trivial cases), and hence we refer to \cite{Bedos95, ALLW-1,ALLW-2} for the analysis of non-separable of F\o lner C*-algebras.
The CAR-algebra, which models the canonical anti-commutation relations, is also in this class (see, e.g., \cite{LM19}), as it is even quasidiagonal (see Definition~\ref{def:FA} and comments below). Finally, we will show that
the traces of the resolvent algebra admit a strong approximation in terms of tracial states on matrix algebras. 

The article is structured as follows. In the next section we extend the notion of F\o lner C*-algebra to the nonunital case and prove some useful results, such as stability under inductive
limits. We also mention some related notions needed later like, e.g., algebraic amenability or the notion of an amenable trace. In Section~\ref{sec:3} we show that the Weyl algebra is a F\o lner C*-algebra. In the last section we study the resolvent algebra and show that it is a F\o lner C*-algebra where any trace is uniform locally finite dimensional, regardless of the dimension of the starting symplectic space.

\paragraph{\textbf{Acknowledgements:}} The first named author would like to acknowledge useful conversations with R.F. Werner in Banff in 2019. Moreover, we thank an anonymous referee for their helpful comments on a previous version of the paper.

\section{F\o lner C*-algebras}
\label{sec:2}

The structure of the resolvent algebra and, in particular, the importance of compact operators in the case of finite dimensional symplectic vector spaces
suggests to introduce the definition of F\o lner C*-algebra in the non-unital case. Concretely, if the underlying symplectic vector space $(X,\sigma)$ is finite dimensional then a characteristic and
useful feature of the associated resolvent algebra is its unique minimal ideal is isomorphic to the 
compact operators (see \cite[Theorem~4.5]{B14} and Section~\ref{sec:4}). 

In the following we generalize the definition of F\o lner C*-algebra (cf. Definition~4.1 in \cite{AL14}) to the non-unital case. Recall that a linear map between C*-algebras
$\varphi\colon \cA\to\cB$ is called \textit{completely positive} (\textit{c.p.}) if it preserves positivity for all matrix amplifications. If $\varphi$ is contractive (i.e., if $\|\varphi\|\leq 1$) we 
abbreviate it by \textit{c.c.p.} If the C*-algebras are unital and $\varphi$ sends the unit of $\cA$ to the unit of $\cB$ we say $\varphi$ is \textit{u.c.p.}
(see \cite{Paulsen08}). Recall that u.c.p. maps are neither multiplicative nor do they preserve orthogonality in general. Define the multiplicative domain of the u.c.p. map $\varphi\colon \cA\to\cB$ as
\begin{equation}\label{eq:multiplicative}
\cA_\varphi:=\{A\in\cA\mid \varphi(A^*A)=\varphi(A)^*\varphi(A)\quadtext{and} \varphi(AA^*)=\varphi(A)\varphi(A)^*\}\;,
\end{equation}
which turns out to be the largest algebra on which $\varphi$ restricts to a *-homomorphism.

\begin{definition}\label{def:FA}
A C*-algebra $\cA$ is a {\em F\o lner C*-algebra} if there exists a net of contractive completely positive (c.c.p.) maps
$\varphi_n\colon\cA\to M_{k(n)}(\C)$ being both
\begin{itemize}
 \item[i)] asymptotically multiplicative, i.e., 
 \begin{equation}\label{eq:mult-2}
\lim_{n} \|\varphi_n(AB)-\varphi_n(A)\varphi_n(B)\|_{2,\mathrm{tr}}=0 \;\; \text{for every} \;\; A,B\in\cA,
\end{equation}
where $\|F\|_{2,\mathrm{tr}}:=\sqrt{\mathrm{tr}(F^*F)}$, $F\in M_{n}(\C)$ and $\mathrm{tr}(\cdot)$ denotes the unique tracial state on a matrix algebra
$M_{n}(\C)$;
 \item[ii)] asymptotically isometric, i.e., 
 \begin{equation}\label{eq:norm}
   \|A\|=\lim_{n}\|\varphi_n(A)\|, \;\; \text{for every} \;\; A\in\cA.
\end{equation}
\end{itemize}
\end{definition}
This definition is analogous to  the general definition of quasidiagonal C*-algebra (see, e.g., Definition~7.1.1 in \cite{bBrown08}).
In fact, a quasidiagonal C*-algebra satisfies the previous two conditions except that asymptotic multiplicativity is defined in terms of the C*-norm $\|\cdot\|$.
Since the $\|\cdot\|_{2,\mathrm{tr}}$-norm is weaker than the C*-norm it is clear that quasidiagonal algebras are a subclass of F\o lner C*-algebras. 
It is also an immediate consequence of Definition~\ref{def:FA} that the class of F\o lner C*-algebras is stable when passing to subalgebras. The same definition above may also be used 
to define F\o lner pre-C*-algebras.

\begin{examples}
 \begin{itemize}
  \item[(i)] The set of compact operators $\cK(\cH)$ on an infinite dimensional Hilbert space is a non-unital quasidiagonal C*-algebra, hence a F\o lner C*-algebra. The c.c.p. maps $\varphi_n\colon\cK(\cH)\to M_{k(n)}(\C)$ are given by compressions $\varphi_n(K):=P_n K P_n$, where $\{P_n\}_{n\in\N}$ is any sequence of non-zero finite rank projections strongly converging to $\1$ and $k(n)=\dim(P_n)$. Note that the quasidiagonalising condition follows from the fact that $\{P_n\}_{n\in\N}$  is an approximate unit for $\cK(\cH)$. 
  \item[(ii)] The Toeplitz algebra $\mathcal{T}$ generated by the unilateral shift on $\ell^2(\N)$ is an example of a F\o lner C*-algebra which is not quasidiagonal. Indeed, note that it is not quasidiagonal, as it has a proper isometry (see~\cite[Proposition~7.1.15]{bBrown08}). Meanwhile, the maps $\varphi_n(A) := P_n A P_n$, where $P_n \in \mathcal{B}(\ell^2(\N))$ projects onto the first $n$ vectors of the canonical orthonormal basis of $\ell^2(\N)$, witness the F\o lner approximation of $\mathcal{T}$.
 \end{itemize}
\end{examples}

The following lemma is the main technical result, and assures that if the C*-algebras are unital one can replace the approximating c.c.p. maps by unital completely positive maps
(see Proposition~\ref{prop:unital}).

\begin{lemma}\label{lem:spreading}
Let $\cA$ be a unital F\o lner C*-algebra and denote by $\varphi_n\colon\cA\to M_{k(n)}(\C)$ a net of c.c.p. maps which are asymptotically multiplicative in the 
$\|\cdot\|_{2,\mathrm{tr}}$-norm and asymptotically isometric in the operator norm. 
Define $e_n:=\varphi_n(\1)\in M_{k(n)}(\C)$ and $P_n:=E_n\left( [1-\varepsilon_n\,,\, 1] \right)$, where $E_n(\cdot)$ denotes the resolution of the identity associated to the 
self-adjoint matrix $e_n$ and $\varepsilon_n > 0$ tends to $0$ when $n$ grows. Then
\[
\lim_{n}\| e_n-P_n\|_{2,\mathrm{tr}}=0\;.
\]
\end{lemma}
\begin{proof}
Let $\varphi_n\colon\cA\to M_{k(n)}(\C)$ be a net of c.c.p. maps as in Definition~\ref{def:FA}.
Consider the self-adjoint matrix $e_n:=\varphi_n(\1)\in M_{k(n)}(\C)$ and denote its spectrum by 
$\left\{\lambda_1^{(n)},\dots, \lambda_{k(n)}^{(n)}\right\}\subset  [0,1]$. The fact that $\varphi_n$ is asymptotically isometric implies that $\|e_n\|\to 1$, while from asymptotic multiplicativity one obtains
\begin{equation}\label{eq:spreading}
\| e_n-e_n^2\|_{2,\mathrm{tr}}^2=\frac{1}{k(n)}\sum_{s=1}^{k(n)}\left( \lambda_s^{(n)}-\left(\lambda_s^{(n)} \right)^2\right)^2\to 0\;.
\end{equation}
Denoting by $E_n(\cdot)$ the spectral resolution of the identity associated to $e_n$, we may split the spectrum of $e_n$ into three regions (close to $0$, intermediate and close to $1$):
\begin{eqnarray*}
\Lambda_0(n)&:=&\mathrm{supp} \left(E_n\right) \cap \left[0\,,\,\varepsilon_n\right],\;
\Lambda_{\mathrm{mid}}(n):=\mathrm{supp} \left(E_n\right) \cap \left(\varepsilon_n\,,\,1-\varepsilon_n \right)\;\;  \mathrm{and}\\
\Lambda_1(n)&:=&\mathrm{supp} \left(E_n\right)\cap  \left[1-\varepsilon_n\,,\, 1 \right].
\end{eqnarray*}
Note that $k(n)=|\Lambda_{0}(n)|+|\Lambda_{\mathrm{mid}}(n)|+|\Lambda_{1}(n)|$ and
since $\|e_n\| \rightarrow 1$ we have that $\Lambda_1(n)\not=\emptyset$ for large $n$.
By definition of the $\|\cdot\|_{2,\mathrm{tr}}$-norm we have
\begin{eqnarray}
 \| e_n-P_n\|_{2,\mathrm{tr}} &=& \frac{1}{k(n)}\left(\sum_{\lambda\in\Lambda_0(n)}\lambda^2 + \sum_{\lambda\in\Lambda_1(n)}(1-\lambda)^2 
                                  + \sum_{\lambda\in\Lambda_{\mathrm{mid}}(n)}\lambda^2\right) \nonumber \\[3mm]                                                               
                              &\leq& \frac{|\Lambda_0(n)|}{k(n)} \frac{\varepsilon_n}{n} + \frac{|\Lambda_1(n)|}{k(n)} \frac{\varepsilon_n}{n} 
                                     + \frac{|\Lambda_{\mathrm{mid}}(n)|}{k(n)}\left(1 - \varepsilon_n\right)^2  \nonumber  \\[3mm]             
                              &\leq& \frac{2\varepsilon_n}{n}+ \frac{|\Lambda_{\mathrm{mid}}(n)|}{k(n)}\left(1 - \varepsilon_n\right)^2. \label{eq:estimate}
\end{eqnarray}
Observe that, from Eq.~(\ref{eq:spreading}), we may, by passing to a subsequence if necessary, assume that 
\[
 \frac{1}{k(n)}\sum_{s=1}^{k(n)}\left( \lambda_s^{(n)}-\left(\lambda_s^{(n)} \right)^2\right)^2 \leq \varepsilon_n \cdot \left(\varepsilon_n-\varepsilon_n^2\right)^2
\]
and, therefore, we have 
\[
 \varepsilon_n \cdot \left(\varepsilon_n-\varepsilon_n^2\right)^2 \geq \frac{1}{k(n)} \sum_{\lambda\in\Lambda_{\mathrm{mid}}(n)}\left(\lambda-\lambda^2\right)^2
              \geq \frac{|\Lambda_{\mathrm{mid}}(n)|}{k(n)} \left(\varepsilon_n-\varepsilon_n^2\right)^2
\]
or, equivalently,
\[
 \frac{|\Lambda_{\mathrm{mid}}(n)|}{k(n)}\leq \; \varepsilon_n.
\]
Using this estimate in the r.h.s. of Eq.~(\ref{eq:estimate}) above we obtain finally the claim, i.e., $\lim_{n}\| e_n-P_n\|_{2,\mathrm{tr}}=0$.
\end{proof}

\begin{proposition}\label{prop:unital}
Let $\cA$ be a unital F\o lner C*-algebra. Then there is a net of u.c.p. maps $\psi_n\colon\cA\to M_{k(n)}(\C)$ which is both
asymptotically multiplicative $\|\cdot\|_{2,\mathrm{tr}}$-norm and asymptotically isometric in C*-norm.
\end{proposition}
\begin{proof}
Let $\varphi_n\colon\cA\to M_{r(n)}(\C)$ be a net of c.c.p. maps witnessing the F\o lner condition of Definition~\ref{def:FA}.
As in the preceding lemma we define $e_n:=\varphi_n(\1)\in M_{r(n)}(\C)$ and denote its spectral resolution by $E_n(\cdot)$. Putting
$P_n:=E_n\left( [1-\frac{1}{n}\,,\, 1] \right)$ note that $P_n e_n$ is an invertible element in $P_n M_{r(n)}(\C) P_n$ and, by functional calculus, we can 
introduce the matrices $f_n:=\left(P_n e_n\right)^{-\frac{1}{2}}$. Finally defining $k(n):=\mathrm{ran} P_n$ we have that the c.c.p. maps given by
\[
 \psi_n\colon\cA\to M_{k(n)}(\C)\;,\quadtext{with} \psi_n(A):= f_n\varphi_n(A) f_n
\]
satisfy the claim. In fact, note that the maps are unital since for any $n$ we have $\psi(\1)=f_n e_n f_n=f_n^2 e_n=P_n$, which is the unit of $P_n M_{r(n)} P_n$.
Asymptotic multiplicativity in $\|\cdot\|_{2,\mathrm{tr}}$-norm follows from Lemma~\ref{lem:spreading} and the fact that, by functional calculus, we have
$\|f_n-P_n\|\to 0$. Finally, note that, in the unital case, the condition of asymptotic isometry can be obtained just by taking direct sums of the
u.c.p. maps $\psi_n$ (see \cite[Proposition~4.2]{AL14} for details).
\end{proof}

\begin{remark}
The previous result shows that for unital C*-algebras c.c.p. maps may be replaced by {\em unital completely positive} (u.c.p.) maps, i.e., completely positive maps sending $\1$ to the matrix unit. Note that the proof requires a different reasoning as in the quasidiagonal case (cf., \cite[Lemma~7.1.4]{bBrown08}) since the asymptotic multiplicative condition in the $\|\cdot\|_{2,\mathrm{tr}}$-norm (i.e., in mean) allows much more general behavior of the spectrum of $\varphi_n(\1)$. In fact, contrary to the quasidiagonal case, the proof of Lemma~\ref{lem:spreading} shows that in general $\Lambda_{\mathrm{mid}}(n)$ may always be non-empty.
\end{remark}

As many C*-algebras used in the physical literature appear as inductive limits of simpler C*-algebras we show next that the class of F\o lner C*-algebras is stable under inductive limits as long as the connecting maps are injective. 

\begin{proposition}\label{pro:inductive}
 The inductive limit of F\o lner C*-algebras with injective connecting maps is also a F\o lner C*-algebra.
\end{proposition}
\begin{proof}
Consider the inductive limit $\cA=\mathop{\lim}\limits_{\longrightarrow} \cA_n$ of F\o lner C*-algebras $\cA_n$ with injective connecting maps. Given $\varepsilon>0$ and a finite subset $\fF\subset \cA_\infty$ we may, without loss of generality, assume that $\fF \subset \cA_{n_0}$ for a sufficiently large $n_0$. Let 
$\varphi_{n_0}\colon\cA_{n_0}\to M_{k(n_0)}(\C)$ be a c.c.p. map witnessing the F\o lner condition for $\cA_{n_0}$, and note that, by Arveson's extension theorem (cf., \cite[Theorem~1.6.1]{bBrown08}),
there is a an extension to a c.c.p. map $\overline{\varphi_{n_0}}\colon\cA \to M_{k(n_0)}(\C)$. It is then clear that the map $\overline{\varphi_{n_0}}$ witnesses the F\o lner property of $\cA$ with respect to $\varepsilon, \cF$, which proves the claim.  
\end{proof}

The class of F\o lner C*-algebras has a rich variety of equivalent characterizations. We begin mentioning the relation with an algebraic version of amenability which 
will be needed in the analysis of the Weyl algebra in the next section. We refer to \cite[Section~3]{ALLW-1} and \cite{Bartholdi2008,Elek03,Gromov99} for additional results and motivation.
We will be interested in the case of *-subalgebras of C$^*$-algebras, but the definition and the results in the references mentioned above are true for arbitrary algebras over arbitrary fields.

\begin{definition}\label{def:alg-amenable}
Let $\fA\subset\cA$ be a *-subalgebra of a C$^*$-algebra $\cA$. We say $\mathfrak{A}$ is \textit{algebraically amenable} if there is 
a net $\left\{V_n\right\}_n$ of nonzero finite dimensional subspaces of $\mathfrak{A}$ satisfying 
\[
 \lim_{n}\frac{\mathrm{dim}(AV_n+V_n)}{\mathrm{dim}(V_n)}=1 \;\; \text{for every} \;\; A\in \mathfrak{A}.
\]
\end{definition}

Next we mention an important relation between algebraic amenability and the
class of F\o lner C$^*$-algebras.
For a complete proof we refer to \cite[Theorem~3.17]{ALLW-2}.
\begin{theorem}\label{teo:alg-amenable-Foelner}
Let $\mathfrak{A}\subset \cA$ be a dense *-subalgebra of a unital separable C$^*$-algebra $\cA$.
If $\mathfrak{A}$ is algebraically amenable, then $\cA$ is a F\o lner C$^*$-algebra.
\end{theorem}

We conclude this brief survey mentioning amenable traces, which are also in close relationship with the class of F\o lner C*-algebras.
Recall that a \textit{tracial state} on a C$^*$-algebra $\cA$ is a positive and normalized functional $\tau\colon\cA\to\C$
that satisfies the usual tracial property $\tau(AB)=\tau(BA)$ for any $A,B\in\cA$. In the next definition we specify 
the subclass of amenable traces (see, e.g., \cite[Chapter~6]{bBrown08}). 

\begin{definition}
Let $\cA\subset\BH$ be a unital C*-algebra. An {\em amenable trace} $\tau$ on $\cA$ is a tracial state on $\cA$ that 
extends to a state $\psi$ on $\BH$ that has $\cA$ in its centralizer, i.e.,
\[
   \tau=\psi|_{\cA}\quad\mathrm{and}\quad \psi(XA)=\psi(AX) \;\; \text{for every} \;\; A\in\cA \;\text{and}\; X\in\BH.
\]
\end{definition}

\begin{remark}
\begin{itemize}
 \item[(i)] Kirchberg uses in \cite[Proposition~3.2]{K94} the name \textit{liftable} trace instead of amenable trace. Moreover, the state $\psi$ in 
the preceding definition is called a {\em hypertrace} in the literature, and the class of F\o lner C*-algebras is also referred to
as \textit{weakly hypertracial} (see \cite{Bedos95} and references therein). The hypertrace $\psi$ can be interpreted as an operator-algebraic generalization 
of the invariant mean of an amenable group.
 \item[(ii)] It is well known (see, e.g., Proposition~6.2.2 in \cite{bBrown08}) that the definition of amenable trace does not depend on the choice of the embedding $\cA\subset\BH$.
 \item[(iii)] Nuclearity and F\o lner type conditions for C*-algebras are, in general, independent notions. 
The relation between the class of nuclear and F\o lner C*-algebras is given by the following statement (see \cite[Corollary~4.8]{AL14}). 
Let $\cA$ be a unital nuclear C*-algebra. Then $\cA$ is a F\o lner C*-algebra iff $\cA$ admits a tracial state. 
Recall also that any trace on a nuclear C*-algebra is amenable (cf. \cite[Proposition~6.3.4]{bBrown08}). We will relate these results when 
analyzing the resolvent algebras in the final section.
 \item[(iv)] However, even though nuclearity and F\o lner type conditions are, in general, independent, they do agree in the class of C*-algebras associated from discrete groups. Indeed, Lance proved in~\cite[Theorem~4.2]{Lance1973} that $C_r^*(G)$ is nuclear iff $G$ is amenable, which happens iff $C_r^*(G)$ is F\o lner in the sense of Definition~\ref{def:FA}.
\end{itemize}
\end{remark}

The following result gives a first relation between unital F\o lner C*-algebras and amenable traces (see \cite{K94} and \cite[Theorem~3.1.6]{Brown06}). 
Observe that we consider unital C*-algebras since the Weyl and resolvent algebras are unital.
Recall from the proof of Proposition~\ref{prop:unital} that if the algebra is unital then any net of u.c.p. maps which is asymptotically multiplicative in mean can be
turned into a net of asymptotically u.c.p. maps which are, in addition, asymptotically isometric in the C*-norm. The following results shows that amenable traces may be approximated in the weak*-topology by matrix traces.

\begin{theorem}\label{teo:Kirchberg}
 Let $\cA$ be a unital C*-algebra and let $\tau$ be a tracial state on $\cA$. Then $\tau$ is an amenable trace if and only if there is a net of u.c.p. maps
 $\varphi_n\colon\cA\to M_{k(n)}(\C)$ which are asymptotically multiplicative in the $\|\cdot\|_{2,\mathrm{tr}}$-norm (cf. Eq.~(\ref{eq:mult-2})) and where 
 $(\mathrm{tr}\circ\varphi_n)(A)\to \tau(A)$ for all $A\in\cA$.  
\end{theorem}

We conclude this section mentioning some additional characterizations of the class of F\o lner C*-algebras in terms of the different notions mentioned before
(see \cite[Theorem~3.8]{ALLW-2} and \cite[Theorem~1.1]{Bedos95}).

\begin{theorem}\label{teo:equivalence}
Let $\cA$ be a unital C*-algebra. Then the following conditions are equivalent.
\begin{itemize}
 \item[(i)] $\cA$ is a F\o lner C*-algebra.
 \item[(ii)] Every faithful representation $\pi\colon\cA\to\cB(\cH)$ satisfies that $\pi(\cA)$ has an amenable trace.
 \item[(iii)] Every faithful and essential representation $\pi\colon\cA\to\cB(\cH)$ (i.e., $\pi(\cA)\cap \cK(\cH)=\{0\}$)
 satisfies that there is a net of positive trace-class operators $\{\rho_n\}_n\subset\cB(\cH)$ with $\mathrm{Tr}(\rho_n)=1$ and such that
\[
\lim_{n}\|\rho_n \pi(A)-\pi(A)\rho_n\|_1=0 \;\; \text{for every} \;\; A\in\cA,
\]
where $\|\cdot\|_1$ denotes the trace-class norm in $\cB(\cH)$.
\end{itemize}
\end{theorem}

\begin{remark}\label{rem:facts}
We mention some additional useful facts around the class of F\o lner C*-algebras in relation to representations:
\begin{itemize}
\item[(i)] In part (iii) of Theorem~\ref{teo:equivalence} above one can choose the $\rho_n$ to have finite dimensional range, and 
also use equivalently the Hilbert-Schmidt norm instead of the trace-class norm.
\item[(ii)] If a nonzero quotient of a unital C*-algebra $\cA$ is F\o lner, then $\cA$ itself
is a F\o lner C*-algebra. In particular, if $\cA$ admits a finite-dimensional representation then it is immediately F\o lner. This property shows, for instance, that the universal
C*-algebra generated by two projections is a F\o lner C*-algebra, as it admits one dimensional representations. This algebra is also 2-subhomogeneous and 
hence also type I (see \cite[IV.1.4.2]{bBlackadar06}).
\item[(iii)] Vaillant introduced in~\cite[Definition~2.1]{Vaillant96} the \textit{F\o lner-Voiculescu condition} on a separable C*-algebra $A$, and related it to growth notions (see also the previous~\cite{KV1992}). It is clear that if $A$ has the F\o lner-Voiculescu condition, in the sense of~\cite{Vaillant96}, then $A$ itself is algebraically amenable, and hence F\o lner in the sense of Definition~\ref{def:FA} by Theorem~\ref{teo:alg-amenable-Foelner} above. However, it is unclear whether these conditions are equivalent. Indeed, note that the F\o lner-Voiculescu condition~\cite{Vaillant96} requires the subspaces $X, Y$ to be contained in $A$, whereas for the F\o lner condition the approximations $\varphi_n$ might not be inner, i.e., coming from subspaces of $A$. This is a subtle difference, very similar to the difference between quasidiagonality and inner quasidiagonality (see~\cite{BlackadarKirchberg2001}).
\end{itemize}
\end{remark}

\section{The Weyl algebra}
\label{sec:3}

Given a non-degenerate symplectic vector space $(X,\sigma)$, the Weyl algebra is defined abstractly as the C*-algebra generated by the unitary elements
(or Weyl elements) $\{W(f)\mid f\in X\}$ satisfying the relations $W(f)^*=W(-f)$ and encoding the canonical commutation relations as
\begin{equation}\label{eq:weyl}
  W(f)W(g)=e^{-\frac{i}{2}\sigma(f,g)} W(f+g)\; f,g\in X\;.
\end{equation}
We denote by $\cW_0(X,\sigma)$ the *-algebra generated by the Weyl elements and by $\cW(X,\sigma)$ the corresponding C*-closure (the \textit{Weyl-algebra}).
Note that the set of Weyl elements forms an (uncountable) basis for $\cW_0(X,\sigma)$.
The Weyl algebra is a unital nonseparable simple C*-algebra, and we refer to \cite{Manuceau68,Petz90,bBratteli02,Lledo04} for proofs, additional results and some applications.

The proof of the fact that the Weyl algebra is a F\o lner C*-algebra exploits the nice multiplicative structure of the Weyl elements that generate this algebra. 

\begin{theorem}
 Let $(X,\sigma)$ be a non-degenerate symplectic vector space. Then the *-algebra $\cW_0(X,\sigma)$ is algebraically amenable. In particular, the Weyl algebra
 $\cW(X,\sigma)$ is a F\o lner C*-algebra.
\end{theorem}
\begin{proof}
To show that the *-algebra $\cW_0(X,\sigma)$ is algebraically amenable we reformulate Definition~\ref{def:alg-amenable} as a local condition. In fact, it is enough to show that
for any $\varepsilon>0$ and any finite $\cF\subset\cW_0(X,\sigma)$ there is a nonzero finite-dimensional subspace $V$ of $\cW_0(X,\sigma)$ such that for all $A\in\cF$ one has
\begin{equation}\label{eq:subspace}
\frac{\dim(AV+V) }{\dim(V)}\leq 1+\varepsilon\;.
\end{equation}
Since any element in $\cF$ is a (finite) linear combination of Weyl elements it suffices to show condition (\ref{eq:subspace}) for $\cF$ being any finite collection of Weyl elements. Consider $\varepsilon>0$ and $\cF:=\{W(g_1),\dots, W(g_n)\}$ for some $g_1,\dots, g_n\in X$. Define the vector space as follows: choose $N> 1/\varepsilon$ and
\[
V:=\mathrm{span} \Big\{ W(k_1 g_1+ \dots +k_n g_n)\mid k_1,\dots k_n \in \{1,\dots N\}\Big\}\;.
\] 
Note that since the Weyl elements are linearly independent we have $\dim(V)= N^n$. Moreover, for any $g_i$, $i=1,\dots, n$, we have
\[
 \dim(W(g_i)V+V)=\dim(V)+N^{n-1}
\]
since only linear combinations of Weyl elements with $k_i=N$ will contribute to the dimension of $ W(g_i)V+V$ not already counted in $V$. Therefore, we have
\[
\frac{\dim(W(g_i)V+V) }{\dim(V)}\leq \left(1+\frac{1}{N}\right)< 1+\varepsilon.
\]

Finally, since $\cW_0(X,\sigma)$ is dense in $\cW(X,\sigma)$ it follows from 
Theorem~\ref{teo:alg-amenable-Foelner} that the Weyl algebra is a F\o lner C*-algebra.
\end{proof}

\begin{remark}\label{rem:approx}
 It is well-known that the Weyl algebra has a unique tracial state. By the results of Section~\ref{sec:2} it is clear that it is an amenable trace and that it can be approximated by a net of states $\{\rho_n\}_n$ in any essential representation. Note nevertheless that, from a physical point of view, this state has not been considered in the physical literature, since it is non-regular and, hence, one lacks the connection with the usual language of quantum fields. We refer to \cite{GH89,GL00} for other results concerning the importance of nonregular states in the presence of quantum constraints.
\end{remark}

\section{The resolvent algebra and finite dimensional approximations}
\label{sec:4}

An alternative C*-algebra also encoding the canonical commutation relations, this time motivated by resolvents of quantum fields, is the so-called {\em resolvent algebra}, introduced by Buchholz and Grundling in \cite{BG08}. This algebra was introduced in order to remedy well-known limitations of the Weyl algebra in the description of quantum systems and quantum dynamics. We refer to \cite{BG08,B14,BG15} for proofs and further results and motivation and to \cite{B17,B18,B20} for additional applications of this algebra to quantum physics. 

Given a non-degenerate symplectic vector space $(X,\sigma)$, 
the resolvent algebra can be defined as an abstract C*-algebra generated by elements $\{ R(\lambda,f)\mid \lambda\in \RS \;\text{and}\; f\in X\}$, 
where for simplicity we denote  $\RS:=\R\setminus\{0\}$ (see Section~3 in \cite{BG08}).
The generators of this algebra satisfy the following relations, which encode
in terms of resolvents the usual properties of fields: normalization, linearity, self-adjointness and resolvent identity. 
For $f,g\in X$ and $\lambda,\nu\in \RS$ the generators satisfy the following six relations:
\begin{eqnarray*}
 R(\lambda,0)&=&-\frac{i}{\lambda} \1\;,\; R(\lambda,f)^*=\R(-\lambda,f)\;,\;  \nu R(\nu\lambda,\nu f)=R(\lambda,f) \;,\\
 R(\lambda,f)-R(\nu,f)&=&i(\nu-\lambda)R(\lambda,f) R(\nu,f)\;,  \\
 R(\lambda,f) R(\nu,g) &=& R(\lambda+\nu,f+g)\Big(R(\lambda,f) + R(\nu,g)+i\sigma(f,g) R(\lambda,f)^2R(\nu, g)\Big) \nonumber
\end{eqnarray*}
In addition, and motivated by the commutation relations of the fields expressed in terms of their resolvents, one requires
\begin{equation}\label{eq:resolvent} 
 \left[ R(\lambda,f),R(\nu,g)\right] = i\sigma(f,g)R(\lambda,f)R(\nu,g)^2 R(\lambda,f)\;.
\end{equation}
Note that these relations imply that the generators $R(\lambda,f)$ are normal operators. We denote the resolvent algebra by
\[
\cR(X,\sigma)=C^*\Big(R(\lambda,f)\mid \lambda\in \RS\;,\; f\in X\Big).
\]

There is a remarkable difference between the Weyl algebra, which is simple and has a unique (non-regular) tracial state, and the resolvent algebra,
which has a rich variety of ideals. This latter fact is crucial to accommodate the description of relevant quantum dynamics as pointed out on \cite[p.~2767]{BG08}.
Moreover, the ideals and the structure of the algebra depend on the dimension of the underlying symplectic space. In fact, if $\mathrm{dim}X<\infty$ then $\cR(X,\sigma)$
is a Type~I C*-algebra, in particular it is nuclear, which means it has a nice representation theory (see \cite[Section~IV.1]{bBlackadar06}).
If $\mathrm{dim}X=\infty$ then the resolvent algebra is merely nuclear.

From Theorem~4.6 in \cite{B14} one can construct traces exploring the maximal ideals of $\cR(X,\sigma)$. 
Let $Z$ be a symplectic subspace of $(X,\sigma)$, i.e., $Z\subset Z^{\perp_\sigma}$, and $\chi$ a pure state on the
Abelian C*-algebra $\cA(Z):=C^*\left(R(\mu,g)\mid \mu\in\RS\;,\; g\in Z \right)$. For each pair $(Z,\chi)$ one can associate a proper maximal ideal 
$\cI=\cI(Z,\chi)$ generated by the sets
\[ 
   \left\{ 
    R(f,\lambda) -\chi(R(f,\lambda))\1\mid \lambda\in\RS\,,\, f\in Z
  \right\}
  \cup
  \left\{ 
   R(g,\mu)\mid \mu\in\RS\,,\, f\in X\setminus Z
  \right\}
  .
\]
In this case one has $\cR(X,\sigma)=\C\1+\cI$, i.e., $\cI$ has codimension~$1$. Tracial states can then be defined by choosing $\tau(D)=0$ if $D\in\cI$.
Note that, since the resolvent algebra is nuclear (cf. \cite[Theorem~3.8~(i)]{B14}) these traces are amenable. We will show next that these traces allow a much 
stronger approximation by means of tracial states of matrix algebras. We will base our reasoning on well-known results from Chapters~3 and 4 in \cite{Brown06}.
Recall that a tracial state $\tau$ on a unital C*-algebra $\cA$ is \textit{uniform locally finite dimensional} if there exist a net of u.c.p. maps
$\varphi_n\colon\cA\to M_{k(n)}(\C)$ satisfying the following two conditions:
\begin{equation}
\label{eq:LFD1} 
   \|\tau-\mathrm{tr}\circ\varphi_n\|_{\cA^*} \to 0 \;,
\end{equation}
where $\|\cdot\|_{\cA^*}$ is the (uniform) norm on the dual of $\cA$; and
\begin{equation}
\label{eq:LFD2} 
   d\left(A,\cA_{\varphi_n}\right) \to 0\;\; \text{for all} \;\; A\in\cA, 
\end{equation}
where $\cA_{\varphi_n}$ is the multiplicative domain of $\varphi_n$ (cf. Eq.~(\ref{eq:multiplicative})), i.e., for any $A$ there are $A_n\in\cA_{\varphi_n}$
with $\|A-A_n\|\to 0$.

Finally, we will exploit the fact that, in general, the resolvent algebra $\cR(X,\sigma)$ is the inductive limit of the resolvent algebras associated with all the
finite dimensional subspaces of $X$.

\begin{theorem}
Given a non-degenerate symplectic vector space $(X,\sigma)$, any trace in the resolvent algebra $\cR(X,\sigma)$ is uniform locally 
finite dimensional, i.e., there exist u.c.p. maps $\varphi_n\colon\cA\to M_{k(n)}(\C)$ satisfying conditions (\ref{eq:LFD1}) and (\ref{eq:LFD2})
above. In particular, the resolvent algebra is a F\o lner C*-algebra and for any tracial state $\tau$ the von Neumann algebra $\pi_{\tau}\left(\cR(X,\sigma)\right)''$ is 
hyperfinite, where $\pi_\tau$ denotes the GNS representation corresponding to $\tau$.
\end{theorem}
\begin{proof}
Consider first the case where $\mathrm{dim}X<\infty$. Then by Theorem~3.8 in \cite{B14} $\cR(X,\sigma)$
is type~I and from Corollary~4.4.4 in \cite{Brown06} we conclude that any trace is uniform locally finite dimensional. This implies that any trace is amenable and 
hence $\cR(X,\sigma)$ is a F\o lner C*-algebra. 

Second, if $\mathrm{dim}X =\infty$ then from Proposition~4.9~(ii) in \cite{BG08} $\cR(X,\sigma)$ is the inductive limit of the net of all $R(S,\sigma)$ where $S \subset X$ 
ranges over all finite-dimensional nondegenerate subspaces of $X$. Moreover, by Proposition~4.9~(i) of \cite{BG08} the connecting maps are injective and so 
$\cR(X,\sigma)$ is a F\o lner C*-algebra as well.

Applying again Corollary~4.4.4 in \cite{Brown06}
to the inductive limit algebra we have that any trace in the resolvent algebra is uniform locally finite dimensional. Finally, since any trace is, in particular, a uniform 
amenable trace it follows from Theorem~3.2.2 in \cite{Brown06} that $\pi_{\tau}\left(\cR(X,\sigma)\right)''$ is a hyperfinite von Neumann algebra.
\end{proof}

\begin{remark} 
\begin{enumerate}
 \item If one is only interested in the C*-algebra there are different ways to show that the resolvent algebra is a F\o lner C*-algebra.
 If fact, since $\cR(X,\sigma)$ is a unital nuclear C*-algebra with a tracial state it is in the F\o lner class (see Remark~\ref{rem:facts}~(iii)).
 Alternatively, from Proposition~4.7 in \cite{B14} the ideal $\cI_c$ generated by commutators of generators of the algebra is proper and its quotient is an 
 Abelian C*-algebra. Therefore, by Remark~\ref{rem:facts}~(ii), the resolvent algebra must also be F\o lner.
 \item Finally, it is an interesting question if the Weyl or resolvent algebras are quasidiagonal. Note that local finite dimensional approximation of the trace $\tau$
 implies that the trace is, in particular, quasidiagonal since one can approximate in norm any pair of elements $A,B\in \cA$ by elements in the multiplicative domain. In fact, from Eq.~(\ref{eq:LFD2}) there exist $A_n,B_n\in\cA_{\varphi_n}$ such that $\|A-A_n\|\to 0$, $\|B-B_n\|\to 0$ and so
 \[
  \varphi_n(AB)\approx \varphi_n(A_nB_n)=\varphi_n(A_n)\varphi_n(B_n)\approx\varphi_n(A)\varphi_n(B)
 \]
 (see \cite[Section~3.4]{Brown06} for more details). From Theorem~\ref{teo:Kirchberg} it follows that the traces are amenable as well. However, the traces we consider for the resolvent algebra are manifestly not faithful, as they vanish in the ideals $\cI_c$. Therefore, should one wish to prove that $\cR(X,\sigma)$ is always quasi-diagonal it is not enough to only consider these traces.
\end{enumerate}

\end{remark}


\renewcommand{\bibname}{References}


\begin{thebibliography}{10}

\bibitem{ALLW-1}
P.~Ara, K.~Li, F.~Lled\'o and J.~Wu, \textit{Amenability of coarse spaces and $\mathbb{K}$-algebras}, 
Bull. Math. Sci. {\bf 8} (2018) 257-306.

\bibitem{ALLW-2}
P.~Ara, K.~Li, F.~Lled\'o and J.~Wu, \textit{Amenability and uniform Roe algebras}, J. Math. Anal. Appl. {\bf 459} (2018) 686-716. 

\bibitem{AL14}
P.~Ara and F.~Lled\'o, \textit{Amenable traces and F\o lner $C^*$-algebras}, 
Expo. Math. {\bf 32} (2014) 161--177.

\bibitem{Arveson94}
W.~Arveson, {\em $C^*$-algebras and numerical linear algebra}, J. Funct. Anal. \textbf{122} (1994) 333--360.

\bibitem{Bartholdi2008}
L.~Bartholdi, {\em On amenability of group algebras I}, Israel J. Math. \textbf{168} (2008), 153--165

\bibitem{Bedos95}
E.~B\'edos, \textit{Notes on hypertraces and $C^*$-algebras},
J. Operator Theory \textbf{34} (1995) 285--306.

\bibitem{Bedos97}
E.~B\'{e}dos, \textit{On F{\o}lner nets, Szeg\"o's theorem and other eigenvalue distribution theorems},
Expo. Math. \textbf{15} (1997) 193--228.
Erratum: Expo. Math. \textbf{15} (1997) 384.

\bibitem{bBlackadar06}
B.~Blackadar, {\em Operator Algebras. Theory of C*-Algebras and von Neumann Algebras}, Springer Verlag, Berlin, 2006.

\bibitem{BlackadarKirchberg2001}
B.~Blackadar and E.~Kirchberg, {\em Inner quasidiagonality and strong NF algebras}, Pacific J. Math. \textbf{198} (2001), 307--329.

\bibitem{bBratteli02}
O.~Bratteli and D.W. Robinson, {\em Operator Algebras and Quantum Statistical
  Mechanics. Vol.~2}, Springer Verlag, Berlin, 2002.

\bibitem{bBrown08}
N.P.~Brown and N.~Ozawa, \textit{$C^*$-Algebras and Finite-Dimensional
Approximations}, American Mathematical Society, Providence,
Rhode Island, 2008.

\bibitem{Brown06}
N.P.~Brown, \textit{Invariant means and finite representation theory of C*-algebras}, Memoirs Am. Math. Soc. {\bf 184} no. 865, (2006), 1-105.

\bibitem{Brown06b}
N.P.~Brown, \textit{Quasidiagonality and the finite section method}, Math. Comput. {\bf 76} (2006), 339-360.

\bibitem{B14}
D.~Buchholz, {\em The resolvent algebra: Ideals and dimension}, J. Funct. Anal. {\bf 266} (2014) 3286-3302.

\bibitem{B17}
D.~Buchholz, 
{\em The Resolvent Algebra for Oscillating Lattice Systems: Dynamics, Ground and Equilibrium States}, 
Commun. Math. Phys. \textbf{353} (2017) 691-716.

\bibitem{B18}
D.~Buchholz, 
{\em The resolvent algebra of non-relativistic Bose fields: observables, dynamics and states},
Commun. Math. Phys. \textbf{362} (2018) 949-981.

\bibitem{B20}
D.~Buchholz, 
\textit{The resolvent algebra of non-relativistic Bose fields: sectors, morphisms, fields and dynamics},
Commun. Math. Phys \textbf{375} (2020) 1159-1199.

\bibitem{BG08}
D.~Buchholz and H.~Grundling, {\em The resolvent algebra: A new approach to canonical
quantum systems}, J. Funct. Anal. {\bf 254} (2008) 2725-2779.

\bibitem{BG15}
D.~Buchholz and H.~Grundling, 
{\em Quantum systems and resolvent algebras}, in: {\em The Message of Quantum Science - Attempts Towards a Synthesis}, 
P. Blanchard, J. Fr\"ohlich Eds., Lect. Notes Phys. Vol.~\textbf{899}, pp. 33-45 Springer Verlag, Berlin 2015.

\bibitem{Elek03}
G.~Elek, \textit{The amenability of affine algebras},
J. Algebra {\bf 264} (2003) 469-478.

\bibitem{Gromov99}
M.~Gromov, {\em Topological Invariants of Dynamical Systems
and Spaces of Holomorphic Maps: I}, 
Math. Phys. Anal. Geom. {\bf 2} (1999) 323-415.

\bibitem{GH89}
H.~Grundling and C.A.~Hurst, {\em A note on regular states and supplementary conditions}, Lett. Math. Phys. {\bf 15} (1988)
205-212; Erratum: Lett. Math. Phys. {\bf 17} (1989) 173-174.

\bibitem{GL00}
H.~Grundling and F.~Lled\'o, {\em Local quantum constraints}, Rev. Math. Phys. {\bf 12} (2000) 1159-1218.

\bibitem{KL14}
R.V.~Kadison and Z.~Liu, \textit{The Heisenberg Relation - Mathematical Formulations}, 
SIGMA, {\bf 10} (2014) 009, 40 pp.

\bibitem{KLT20}
R.V.~Kadison, Z.~Liu and A.~Thom, \textit{A note on commutators in algebras of unbounded operators}
Expo. Math. {\bf 38} (2020) 232-239.

\bibitem{K94}
E.~Kirchberg, \textit{Discrete groups with Kazhdan’s property T and factorization property
are residually finite}, Math. Ann. {\bf 299} (1994) 551-563.

\bibitem{KV1992}
E.~Kirchberg and G.~Vaillant, {\em On C*-algebras having linear, polynomial and subexponential growth}
Invent. Math. \textbf{108}, (1992) 635--652.

\bibitem{Lance1973}
C.~Lance, {\em On nuclear C*-algebras}, J. Funct. Anal. \textbf{12} (1973) 157--176. 

\bibitem{Lledo04}
F.~Lled\'o, {\em Massless relativistic wave equations and quantum field theory}, Ann. H. Poincaré {\bf 5} (2004) 607-670. 
  
\bibitem{Lledo13}
F.~Lled\'o,
\textit{On spectral approximation, F\o lner sequences and crossed products},
J. Approx. Theory {\bf 170} (2013) 155--171.

\bibitem{LM19}
F.~Lled\'o and D.~Mart\'inez,
\textit{Notions of infinity in quantum physics}, In, {\em Classical and Quantum Physics}, G. Marmo {\em et al.} eds,
Springer Proceedings in Physics Vol.~229., Springer, Cham, 2019.

\bibitem{Manuceau68}
J.~Manuceau,  \textit{C*-alg\`ebre de relations de commutation}, Ann. Inst. H. Poincar\'e Ser. A \textbf{8} (1968) 139-161.

\bibitem{Paulsen08} V.~Paulsen, \textit{Completely Bounded Maps and Operator Algebras}, Cambridge University Press, Cambridge, 2002.

\bibitem{Petz90} D.~Petz, \textit{An Invitation to the Algebra of Canonical Commutation Relations}, 
Leuven University Press, Leuven, 1990.

\bibitem{Vaillant96}
G.~Vaillant, {\em F{\o}lner conditions, nuclearity, and subexponential growth in C*-algebras},
J. Funct. Anal. {\bf 141} (1996) 435--448.

\bibitem{Wielandt49}
H.~Wielandt, \textit{\"Uber die Unbeschr\"anktheit der Schr\"odingerschen Operatoren der Quantenmechanik}, 
Math. Ann. \textbf{121} (1949) 21.

\bibitem{Wintner47}
A. Wintner, \textit{The unboundedness of quantum mechanical matrices}, Phys. Rev.  
\textbf{71} (1947) 738-739.

\end{thebibliography}
\end{document}